\documentclass[11pt]{amsart}
\usepackage[margin=1in]{geometry}
\usepackage{hyperref}
\usepackage{amsmath}
\usepackage{amssymb}
\usepackage{amsthm}
\usepackage{graphicx}
\usepackage{graphics}
\usepackage{tikz-cd}
\usepackage{cite}
\usepackage[mathscr]{euscript}
\usepackage{shuffle}
\usepackage{physics}
\usepackage{enumerate}
\usepackage{enumitem}

\theoremstyle{plain}
\newtheorem{Th}{Theorem}[section]
\newtheorem{Lemma}[Th]{Lemma}

\newtheorem{Cor}[Th]{Corollary}
\newtheorem{Prop}[Th]{Proposition}

\theoremstyle{definition}
\newtheorem{Def}[Th]{Definition}

\newtheorem{Rem}[Th]{Remark}
\newtheorem{?}[Th]{Problem}

\newcommand{\K}{\mathbb{K}}

\newcommand{\Z}{\mathbb{Z}}

\DeclareMathOperator{\id}{id}

\DeclareMathOperator{\ad}{ad}

\DeclareMathOperator{\conc}{conc}

\DeclareSymbolFont{rsfs}{U}{rsfs}{m}{n}
\DeclareSymbolFontAlphabet{\mathscrsfs}{rsfs}

\begin{document}
\title{Double shuffle relations and double antipodes}
\author{Muze Ren}
\address{Institut de Recherche Math\'ematique Avanc\'ee, UMR 7501, Universit\'e de Strasbourg et CNRS, 7
rue Ren\'e Descartes, 67000 Strasbourg, France}
\email{muze.ren@unige.ch}
\maketitle
\begin{abstract}
In this note, we prove that for every Lie series $\psi$ with no terms of degree less than 3, the relation $$[\psi(x,y),x]+[\psi(-x-y,y),-x-y]=0$$ is equivalent to $S_*(\psi_*)=-\psi_*$, where $\psi_*$ denotes the regularization of $\psi$ and $S_*$ is the harmonic antipode. The proof relies on the calculations of the harmonic antipode $S_*$ and the shuffle antipode $S$. As a consequence, we prove that every $\psi$ in Racinet's double shuffle Lie algebra $\mathfrak{dmr}_0$ satisfies the relation $[\psi(x,y),x]+[\psi(-x-y,y),-x-y]=0$. We further prove that $\mathfrak{dmr}_0$ injects into the symmetric Kashiwara--Vergne Lie algebra $\mathfrak{krv}^{\mathrm{sym}}_2$ of Alekseev and Torossian.
\end{abstract}

\section{Introduction}
\subsection{Multiple zeta values (MZVs) and double shuffle Lie algebra}
Multiple zeta values are real numbers 
\begin{equation*}
    \zeta(s_1,\ldots,s_l)=\sum_{n_1>n_2>\ldots>n_l\ge 1}\frac{1}{n^{s_1}_1n^{s_2}_2\ldots n^{s_l}_l},
\end{equation*}
with $s_i\ge 1$ for $1\le i\le l$ and $s_1\ge 2$. These numbers generalize the special values of the Riemann zeta function at positive integers. In the case of $l=2$, they were already studied by Euler back to the eighteenth century. MZVs were rediscovered independently by D.~Zagier and J.~\'Ecalle. MZVs satisfy three distinguished set of relations, see for example \cite{IKZ}, namely the stuffle relation, shuffle relation and regularization relations, these relations have deep connections with the theory of motives \cite{Deligne,DG, Brown2012}. In the work \cite{Racinet2002}, G. Racinet introduced the Lie algebra $\mathfrak{dmr}_0$ and used it to study formally those three sets of relations. 

\subsection{Grothendieck--Teichm\"uller Lie algebra and a variant of one defining equation} The Grothendieck--Teichm\"uller Lie algebra $(\mathfrak{grt}_1)$ plays a central role in many branches of mathematics. The following equation
\begin{equation}\label{eq:hexgon}
   [\psi(A,B),B]+[\psi(A,C),C]=0,\quad \text{for}\quad A+B+C=0, 
\end{equation}
was introduced by V.~Drinfeld in the definition of the Grothendieck--Teichm\"uller Lie algebra $\mathfrak{grt}_1$ , see equation (5.19) in \cite{Drinfeld1991}, and by Y.~Ihara in the study of the symmetric outer derivation Lie algebra, see equation (2.2.2) in \cite{Ihara1992}. Using the skew-symmetry condition $\psi(A,B)=-\psi(B,A)$, one obtains a variant of equation \eqref{eq:hexgon}
\begin{equation}\label{eq:fhexagon}
    [\psi(B,A),B]+[\psi(C,A),C]=0,\quad \text{for}\quad A+B+C=0.
\end{equation}
It was proved by H.~Furusho \cite{Furusho10} that the pentagon equation implies the infinitesimal hexagon equation $\psi(C,A)+\psi(B,C)+\psi(A,B)=0$ for $A+B+C=0$ and skew-symmetry condition $\psi(A,B)=-\psi(B,A)$, combining with Proposition 5.7 in \cite{Drinfeld1991}, the pentagon equation imply the equation \eqref{eq:fhexagon}. H.~Furusho also proved that $\mathfrak{grt}_1\hookrightarrow \mathfrak{dmr}_0$ in \cite{Furusho2011}.

\subsection{Kashiwara--Vergne Lie algebra and the special derivation property} The Kashiwara--Vergne Lie algebra $\mathfrak{krv}_2$ was introduced in \cite{Alekseev2012} by A.~Alekseev and C.~Torossian in their study of Kashiwara--Vergne conjecture \cite{Alekseev2012}, they proved that $\mathfrak{grt}_1\hookrightarrow \mathfrak{krv}_2$. They introduce the tangential derivation Lie algebra $\mathfrak{tder}_2$ consisting of derivations of the free Lie algebra of two generators $\mathbb{L}(x,y)$ that sends $x\mapsto [x,a]$ and $y\mapsto [y,b]$, the special derivations $\mathfrak{sder}_2=\{(a,b)\in \mathfrak{tder}_2\mid [x,a]+[y,b]=0\}$. A related inertia condition is used by Enriquez--Furusho in the study of $\mathfrak{dmr}_0$, for $z=-x-y$ and a Lie series $\psi$, the inertia relation of $\psi$ is the existence of a companion $b_{\psi}$ satisfying $[x,\psi]+[z,b_{\psi}]=0$. The map $I$ in \eqref{map:EF} maps inertia condition to the special derivation condition. The following Theorem \ref{th:EFSH} is stated by J.~ \'Ecalle without proof using his mould theoretic language, see equation (3.64) in \cite{Ecalle}. In recent papers \cite{Schneps2025, kawamura2025, EF4}, the proofs are provided independently, we use the formulation in \cite{EF4}. 

\begin{Th}[{\cite[Corollary 0.31(c), Theorem 0.38]{EF4}}]\label{th:EFSH}
  For $\psi\in \mathfrak{dmr}_0$, there exists unique $b_{\psi}\in \mathbb{L}^{\ge 3}(x,y)$ such that 
\begin{equation}\label{eq:innertia}
[\psi(x,y),x]+[b_{\psi}(x,y),-x-y]=0,
\end{equation}
and the map $I$ is an injective graded Lie algebra morphism
\begin{equation}\label{map:EF}
    I: \mathfrak{dmr}_0\to \mathfrak{krv}_2, \psi(x,y)\mapsto (\psi(z,y)-b_{\psi}(z,y),\psi(z,y)).
\end{equation}
\end{Th}

The variables $x,y,z$ correspond to $e_0,e_1, e_{\infty}$ respectively, in their notation. The explicit form of $b_{\psi}$ is unknown.  The connection with the formulation of \cite{Ecalle} is explained in Appendix of \cite{Schneps2012} and the connection with \cite{kawamura2025,Schneps2025} is explained in   Subsection 16.1 in \cite{EF4}.

\subsection{Main results}

Our first result gives two equivalent characterizations of the property of being harmonic antipode anti invariant for a Lie series with no terms of degree less than 3.

\begin{Th}\label{th: harmonic_antipode}
    For every $\psi\in \mathbb{L}^{\ge 3}(x,y)$, the following two conditions are equivalent.
    \begin{enumerate}
        \item $S_*(\psi_*)=-\psi_*$.
        \item $[\psi(x,y),x]+[\psi(-x-y,y),-x-y]=0$.
    \end{enumerate}
\end{Th}

The first characterization is from the Hopf algebra perspective and this is directly related to being harmonic coproduct primitive, namely $\Delta_*(\psi_*)=1\otimes \psi_*+\psi_*\otimes 1$ implies $S_*(\psi_*)=-\psi_*$. The second characterization is related to the inertia condition in $\mathfrak{dmr}_0$, the special derivation property in $\mathfrak{krv}_2$ and the equation \eqref{eq:fhexagon} in $\mathfrak{grt}_1.$ 

\begin{Rem}
    The equation $[\psi(x,y),x]+[\psi(-x-y,y),-x-y]=0$ is different from the equation \eqref{eq:hexgon}, it is the equation \eqref{eq:fhexagon}, we thank Hanamichi Kawamura for this remark.
\end{Rem}

The second result is an application of the first Theorem \ref{th: harmonic_antipode}.

\begin{Th}[{\cite[Corollary 1]{Schneps2025}}]\label{th:double_shuffle}
    For every $\psi\in \mathfrak{dmr}_0$, it satisfies the equation \[[\psi(x,y),x]+[\psi(-x-y,y),-x-y]=0.\]
\end{Th}

This Theorem \ref{th:double_shuffle} provides an explicit companion $b_{\psi}=\psi(-x-y,y)$ in Theorem \ref{th:EFSH} by the uniqueness of $b_{\psi}$. It is a fundamental property of $\mathfrak{dmr}_0$, a variant of the  Theorem \ref{th:double_shuffle} appears in \cite[Corollary 1]{Schneps2025}, which is proved using mould theoretic language and literature. We use the standard Hopf algebra property to give an alternative proof here which are easier to generalize to other situations and conceptual clear.

As a consequence, we prove that $\mathfrak{dmr}_0$ not only injects into the Kashiwara--Vergne Lie algebra, it is contained in the symmetric Kashiwara--Vergne Lie algebra $\mathfrak{krv}^{\mathrm{sym}}_2$ by a different morphism. 

\begin{Th}\label{th:symmetr}
    The following map $J$ is an injective Lie algebra map
    \begin{equation*}
        J: \mathfrak{dmr}_0\to \mathfrak{krv}^{\mathrm{sym}}_2,\quad \psi(x,y)\mapsto (-\psi(x,-x-y),-\psi(y,-x-y)).
    \end{equation*}
\end{Th}

$\mathfrak{krv}_2$ is a graded Lie algebra, we denote by $(\mathfrak{krv}_2)_{\ge 3}$ the degree equal or bigger than $3$ part of $(\mathfrak{krv}_2)_{\ge 3}$. We first show that the symmetric group $S_3$ acts on $(\mathfrak{krv}_2)_{\ge 3}$ in Theorem \ref{th:S_3KV}. Then the proof of this theorem is by noticing that the image of $I$ with the explicit companion $b_{\psi}$ provided by the Theorem \ref{th:double_shuffle} is fixed by the automorphism on $(\mathfrak{krv}_2)_{\ge 3}$ induced by the $x$ and $z$ permutation. The symmetric Kashiwara--Vergne Lie algebra is the fixed point of the automorphism on $(\mathfrak{krv}_2)_{\ge 3}$ induced by $x$ and $y$ permutation, therefore it suffices to apply the automorphism on $(\mathfrak{krv}_2)_{\ge 3}$ induced by $y$ and $z$  to the image of $I$, the composition of the automorphism with $I$ is $J$.

\begin{Rem}
We mention some works that are related.

In Section 6.5 of \cite{Brown_depth}, F.~Brown shows that the shuffle and harmonic antipodes build up dihedral symmetry in each part of the depth-graded version of $\mathfrak{dmr}_0$. Taking the leading depth element of the formula in Proposition \ref{Prop:formula} and passing to the depth graded setting, the cyclic invariance immediately following the equation (6.13) in  \cite{Brown_depth} for the depth graded $\mathfrak{dmr}_0$.

The author is informed that related questions concerning the harmonic antipode and the companion in Theorem \ref{th:EFSH} are being studied independently by Hanamichi Kawamura and Takumi Maesaka in their forthcoming paper.
\end{Rem}

{\bf Organization of the paper.} This text has three sections after the introduction. In the first section \ref{sec:preliminary}, we recall the basic notations and study some properties of shuffle antipode and harmonic antipode. In the second section \ref{sec:proof_of_first}, we prove Theorem \ref{th: harmonic_antipode}. In the third section \ref{sec:dmr_krv}, we recall the definitions and properties of $\mathfrak{dmr}_0$ and $\mathfrak{krv}_2$. We first prove Theorem \ref{th:double_shuffle}, and then establish the $S_3$ action of $(\mathfrak{krv}_2)_{\ge 3}$, after that we prove the Theorem \ref{th:symmetr}.

\vspace{0.5cm}

{\bf Acknowledgements.}   The author thanks Benjamin Enriquez and Hidekazu Furusho for explaining their work, helpful comments and suggestions. The author also thanks Anton Alekseev, Guillaume Laplante-Anfossi, Leila Schneps and Thomas Willwacher for support and help during the preparation of this work. The author also thanks Hanamichi Kawamura and Takumi Maesaka for helpful discussions and for sharing their perspective. The author is supported by the SNSF postdoc mobility grant  $P500PT\_230340$.

\section{Preliminary}\label{sec:preliminary}
We fix a field $\mathbb{K}$ of characteristic 0. Let $A$ denote the Hopf algebra of formal noncommutative polynomials $A := 
    \big( \mathbb{K}\langle\langle x,y\rangle\rangle, \conc, \Delta_\shuffle, \varepsilon,S \big),$ where $\conc$ is the concatenation product and
\begin{equation*}
    \begin{split}
        & \Delta_\shuffle(x) = x \otimes 1 + 1 \otimes x, \quad \varepsilon(x) = 0, \quad S(x)=-x, \\
        & \Delta_\shuffle(y) = y\otimes 1 + 1 \otimes y, \quad \varepsilon(y) = 0, \quad S(y)=-y, \\
    \end{split}
\end{equation*}
where $S$ is the shuffle antipode. $A$ is graded, the grading of $A$ is defined by $\deg(x)=1$ and $\deg(y)=1$. Let $\mathbb{L}(x,y)$ denote the completed graded free Lie subalgebra of $A$ generated by $x,y$ and $\mathbb{L}^{\ge n}$ denotes the part that has no terms of degree lower than $n$.

For a word $w$ and series $P$ in $A$, we denote $(P|w)$ for the coefficient of $w$ in $P$, for every $P$ in $A$, it has the following decompositions
\begin{equation*}   P=\varepsilon(P)+P_xx+P_yy=\varepsilon(P)+xP^x+yP^y,
\end{equation*}
where $P_x, P_y,P^x,P^y$ denote the remaining parts of $P$ that ending in x, y and beginning in $x,y$.

Let $Y_i, i\ge1, i\in \Z$ be the free generators of degree $i$ and Let $\mathbb{K}\langle Y\rangle$ denote the free algebra generated by them, it is equipped with the bialgebra structure, the product is the concatenation, the harmonic coproduct is the algebra morphism determined by
\begin{equation}\label{eq:coproduct}
    \Delta_{*}(Y_m)=Y_m\otimes 1+1\otimes Y_m-\sum_{r+s=m,r,s\ge 1}Y_r\otimes Y_s,
\end{equation}
the counit $\varepsilon_*$ is defined to be $\varepsilon_*(1)=1$ and $\varepsilon_*(Y_i)=0$, then $(\mathbb{K}\langle Y\rangle,\conc, \Delta_*,\varepsilon_*)$ is a connected $\mathbb{N}$ graded bialgebra. By the Cartier-Quillen-Milnor-Moore theorem \cite{CMM}, it is a Hopf algebra with antipode $S_*$. Let $\mathbb{K}\langle\langle Y\rangle \rangle$ denote the degree completion of $\mathbb{K}\langle Y\rangle$. The map that sends $Y_m$ to $x^{m-1}y$ induces an injective algebra map from $\mathbb{K}\langle\langle Y\rangle\rangle$ to $A$, the image is $\mathcal{W}:=\mathbb{K}\oplus Ay$. In what follows, we use the identification of $\mathbb{K}\langle\langle Y\rangle\rangle$ with its image $\mathcal{W}$ in $A$, let $\pi_Y:A\to \mathcal{W}$ denote the projection, $\pi_Y(P):=\varepsilon(P)+P_yy$.

For any $\alpha\in A$, we define the regularization map,
\begin{equation}\label{eq:regulirization}
 (-)_{*}: A\to \K\langle\langle Y\rangle\rangle; \alpha\mapsto (\alpha)_*:=\pi_Y(\alpha)+\sum_{n\ge 1}\frac{(\alpha\mid x^{n-1}y)}{n}y^n. 
\end{equation}

\begin{Rem}
The sign convention of the coproduct $\Delta_*$ \eqref{eq:coproduct} and $(-)_*$ \eqref{eq:regulirization} are consistent with the ones in Section 0.3 in \cite{EF4}, the notation $e_0,e_1$ in \cite{EF4} corresponds to our $x,y$.    
\end{Rem}

\subsection{Properties of harmonic antipode}
In this subsection, we study the harmonic antipode $S_*$ and the equality $S_*(\psi_*)=-\psi_*.$  Let us first introduce the map
\begin{equation}
    Q:=(-)_*\circ S-S_*\circ (-)_*,
\end{equation}
which measures the failure of the regularization map to intertwine the shuffle and harmonic antipode.

\begin{Lemma}\label{lem:Q_equivalence}
    Suppose $\psi\in \mathbb{L}^{\ge 3}(x,y)$, 
    \[
    \psi\in\ker Q
\quad\Longleftrightarrow\quad
    S_*(\psi_*)=-\psi_*.
\]
\end{Lemma}

\begin{proof}
    For $\psi\in \mathbb{L}(x,y)$, we have the relation $S(\psi)=-\psi$. Suppose $S_*(\psi_*)=-\psi_*$, then $Q(\psi)=S(\psi)_*-S_*(\psi_*)=-\psi_*-(-\psi_*)=0.$ For another direction, suppose that $\psi\in \ker Q$, then $S_*(\psi_*)+\psi_*=0$.
\end{proof}

In the following, we compute the defect $Q$ explicitly.

\begin{Prop}
We denote by $z=-x-y$.
\begin{enumerate}
    \item  For every $m\ge 1$, we have
    \begin{equation*}
        S_*(Y_m)=(-1)^m z^{m-1}y.
    \end{equation*}
    \item For every $n_1,n_2,\ldots,n_r\ge 1$, we have
    \begin{equation}\label{eq:antipode formula}
        S_*(Y_{n_1}\ldots Y_{n_r})=(-1)^{n_1+\ldots+n_r}z^{n_r-1}y\ldots z^{n_1-1}y
    \end{equation}
\end{enumerate}
   
\end{Prop}

\begin{proof}
    We introduce the formal power series
    \begin{equation*}
        \mathcal{G}(t)=1-\sum_{m\ge 1}Y_mt^m,
    \end{equation*}
    by the coproduct formula, we have
    \begin{equation*}
        \Delta_*(\mathcal{G}(t))=\mathcal{G}(t)\widehat{\otimes} \mathcal{G}(t),
    \end{equation*}
    therefore $\mathcal{G}(t)$ is group like, we have
    \begin{equation*}
        \mathcal{G}(t)=1-t(1-tx)^{-1}y=(1-tx)^{-1}(1-tx-ty)=(1-tx)^{-1}(1+tz),
    \end{equation*}
    therefore we have
    \begin{equation}\label{eq:antipode-formula_1_proof}
        \mathcal{G}(t)^{-1}=(1+tz)^{-1}(1-tx)=1+\sum_{m\ge 1}(-1)^{m+1}z^{m-1}yt^m
    \end{equation}
and since the antipode sends a group-like element to its inverse
\begin{equation}\label{eq:antipode-formula_2_proof}
    S_*(\mathcal{G}(t))=1-\sum_{m\ge 1}S_*(Y_m)t^m=\mathcal{G}(t)^{-1},
\end{equation}
by comparing the coeffient of the equation \eqref{eq:antipode-formula_1_proof} and \eqref{eq:antipode-formula_2_proof}, we derive the formula. The second formula is because the antipode is an anti-algebra map.
\end{proof}

Let $\mathrm{rev}$ be the linear operation on $A$ that reverse the order of the words in each polynomial.

\begin{Lemma}\label{lem: reverse}
    If $a\in \mathbb{L}(x,y)$ is homogeneous of degree $n$, then
    \begin{enumerate}
        \item $\mathrm{rev}(a)=(-1)^{n-1}a$\\
        \item $(a\mid u\shuffle v)=0$, for any $u,v\in A$ non empty of degree $\ge 1$.
    \end{enumerate}
\end{Lemma}

\begin{proof}
For the first claim, as $a\in \mathbb{L}(x,y)$, we have $S(a)=-a$. For a homogeneous word of degree $n$, by definition of $S$, we have $S(a)=(-1)^n\mathrm{rev}(a)$, therefore we have $(-1)^n\mathrm{rev}(a)=-a$.

The second one follows from $a\in \mathbb{L}(x,y)$ is primitive.

\end{proof}

\begin{Prop}Let $a\in \mathbb{L}(x,y)$ be homogeneous of degree  $n\ge 2$, then
\begin{equation}\label{eq:formula-piy}
    S_*(\pi_Y(a))=-a^y(z,y)y.
\end{equation}
    
\end{Prop}

\begin{proof}
For any word $w$ in A of degree $n$ ending in $y$, we could write it uniquely as
    \begin{equation*}
        w=Y_{n_1}\ldots Y_{n_r},
    \end{equation*}
    let $p$ be the word determined by $\mathrm{rev}(w)=yp$, by the explicit formula \eqref{eq:antipode formula}, we have
    \begin{equation*}
        S_{*}(w)=(-1)^np(z,y)y.
    \end{equation*}
    By Lemma \ref{lem: reverse}, we have that \begin{align*}
        (\pi_Y(a)\mid w)=(a\mid w)&=(\mathrm{rev}(a)\mid \mathrm{rev}(w))=(-1)^{n-1}(a\mid \mathrm{rev}(w))=(-1)^{n-1}(a\mid yp)\\
        &=(-1)^{n-1}(a^y\mid p)
    \end{align*}

Multiplying this coefficient by the sign $(-1)^n$ gives $-(a^y\mid p)$, as $w$ ranges over all words of degree $n$ ending $y$, the word $p$ ranges bijectively over all words of length $n-1$, therefore we get the above formula.
\end{proof}

\begin{Prop}\label{prop:formula_Q}
    Let $a\in \mathbb{L}^{\ge 3}(x,y)$ be homogeneous of degree $n$,
    \begin{equation*}
        Q(a)=\left(a^y(z,y)-a_y(x,y)-\frac{1+(-1)^n}{n}(a\mid x^{n-1}y)y^{n-1}\right)y
    \end{equation*}
\end{Prop}

\begin{proof}
As $y$ is primitive, we have $S_*(y)=-y$, it follows that $S_*(y^n)=(-1)^ny^n$. By the formula \eqref{eq:formula-piy}, we have
\begin{align*}
    &S_*(a_*)=S_*(\pi_Y(a)+\frac{(a\mid x^{n-1}y)}{n}y^n)=S_*(\pi_Y(a))+S_*(\frac{(a\mid x^{n-1}y)}{n}y^n)\\&=-a^y(z,y)y+\frac{(a\mid x^{n-1}y)}{n}(-1)^ny^n.
\end{align*}
By definition we have
\begin{align*}
-a_*=-a_yy-\frac{(a\mid x^{n-1}y)}{n}y^n.  
\end{align*}
Therefore we have the formula for $Q(a)=S(a)_*-S_*(a_*)=-a_*-S_*(a_*)$.    
\end{proof}

\begin{Prop}\label{Prop:formula}
Let $a\in \mathbb{L}^{\ge 3}(x,y)$ be homogeneous of degree $n$, then the following two conditions are equivalent
\begin{enumerate}
    \item $S_*(a_*)=-a_*$.
    \item $a^y(z,y)-a_y(x,y)=\frac{1+(-1)^n}{n}(a\mid x^{n-1}y)y^{n-1}$.
\end{enumerate}
\end{Prop}

\begin{proof}
By the Lemma \ref{lem:Q_equivalence} and Proposition \ref{prop:formula_Q}, $S_*(a_*)=-a_*$ is equivalent to $Q(a)=0$, which is equivalent to 
$\left(a^y(z,y)-a_y(x,y)-\frac{1+(-1)^n}{n}(a\mid x^{n-1}y)y^{n-1}\right)y=0$. And this is equivalent to $a^y(z,y)-a_y(x,y)-\frac{1+(-1)^n}{n}(a\mid x^{n-1}y)y^{n-1}=0$ because right multiplication by $y$ is injective.
\end{proof}

\subsection{Auxiliary results on free Lie algebra}
We prepare some useful facts about free Lie algebra for later use, and we present the proofs for convenience.

Fix two letters $u< v$ and order words of fixed length lexicographically.
\begin{Lemma}\label{lem:order}
Let $0\ne H\in \mathbb{L}(u,v)$ be homogeneous of degree $N\ge 2$, the lexicographically smallest word occurring in $H$ begins with $u$ and ends with $v$.    
\end{Lemma}

\begin{proof}
 Let $\omega$ be the least word with respect to lexicographic order such that $(H\mid \omega)\ne 0$, we first show that $w$ cannot begin with $v$. We prove this by contradiction.

 Suppose that $\omega$ starts with $v$, we write $\omega=v^rq, r\ge 1$, where $q$ is empty or its first letter is $u$, if $q$ is empty, we have $v^{N-1}\shuffle v=Nv^N$, so by the primitive of $H$ and Lemma \ref{lem: reverse}, we have $(H\mid \omega)=\frac{1}{N}(H\mid v^{N-1}\shuffle v)=0$. Assume that $q$ is nonempty and put $p=v^{r-1}q$, then we have
 \begin{equation*}
     0=(H\mid v\shuffle p)=r(H\mid \omega),
 \end{equation*}
 because inserting $v$ in $q$ will resulting a word that begin with $v^{r-1}$ smaller than $\omega$ which is not detected by $H$ by the assumption of $\omega$ and  $(H\mid \omega)=0$ is a contradiction.

 Now we show that $\omega$ can not ends in $u$. We prove by contradiction and assume that $\omega$ ends in $u$. We could write $\omega=qu^r$ with $r\ge 1$ and $q$ is either empty or ends in $v$. If $q$ is empty, then $(H\mid \omega)=\frac{1}{N}(H\mid u\shuffle u^{N-1})=0$ which is a contradiction. Otherwise, we put $p=qu^{r-1}$. In the shuffle of $p \shuffle u$, if $u$ is inserted in the slot that beginning
immediately before the terminal block $u^{r-1}$ and ending immediately after it will all give the same word $\omega$. Any insertion earlier in the word that occurs before some later letter $v$ (such a $v$ always exist because $q$ ends in $v$) will resulting a word lexicographically smaller than $\omega$, because $\omega$ is the smallest in $H$, it makes no contribution to the pairing, therefore
\begin{equation*}
    0=(H\mid p\shuffle u)=r(H\mid \omega)
\end{equation*}
a contradiction, and therefore $\omega$ ends with $v$.
\end{proof}

\begin{Cor}\label{cor:zero}
    Let $H\in \mathbb{L}(u,v)$ be homogeneous of degree at least two, if every word occurring in $H$ begins and ends in the same letter, then $H=0$.
\end{Cor}

\begin{proof}
  Suppose that $H$ is not equal to 0, then by the Lemma \ref{lem:order}, the least word begins with $u$ and ends in $v$.  
\end{proof}

\begin{Lemma}\label{lem:mix_boundary}
Let $p\in \mathbb{L}(u,v)$ be an element of homogeneous degree $N\ge 4$, suppose that every word of length $N$ that begins and ends in different letters has the same coefficient $\lambda$ in $p$, then $\lambda=0$ and $p=0$.
 \end{Lemma}

\begin{proof}
Let $\kappa:=(p\mid uv^{N-2} u)$ be the coefficient of $p$ before the word $uv^{N-2}u$, by the Lemma \ref{lem: reverse}, we have
\begin{align*}
    &(p\mid u\shuffle v^{N-2} u)=\kappa+(N-1)\lambda=0\\
    &(p\mid v\shuffle uv^{N-3} u)=2\lambda+(N-2)\kappa=0
\end{align*}
by the first equality $\kappa=-(N-1)\lambda$, we substitute it into the second one we have $-N(N-3)\lambda=0$, for $\mathbb{K}$ of characteristic zero, we have $\lambda=0$, therefore $p$ contains only the word that begins and ends in the same letter, by the Corollary \ref{cor:zero}, $p=0$.
\end{proof}

\section{Proof of the Theorem \ref{th: harmonic_antipode}} \label{sec:proof_of_first}
In this section, we prove the Theorem \ref{th: harmonic_antipode}. It suffices to prove the statement for the homogeneous element, we assume that $a$ in $\mathbb{L}(x,y)$ is homogeneous of degree $n\ge 3$ throughout this section.

\subsection{From $(1)$ to $(2)$}
 We introduce the free generators $v,u$ and denote $\sigma$ the algebra morphism of $\mathbb{K}\langle\langle u,v\rangle\rangle$ that interchange $u$ and $v$, set
    \begin{equation}\label{eq:F_uv}
        F(u,v):=a(v,-u-v),
    \end{equation}
    $F$ is a homogeneous element of degree $n$, by the decomposition $a=a_xx+a_yy$, we have
    \begin{equation*}
        F(u,v)=a_x(v,-u-v)v+a_y(v,-u-v)(-u-v),
    \end{equation*}
    so we have $F_u(u,v)=-a_y(v,-u-v)$, and similarly by the decomposition $a=xa^x+ya^y$, we have
    \begin{equation*}
        F(u,v)=va^x(v,-u-v)+(-u-v)a^y(v,-u-v),
    \end{equation*}
    so that we have $F^u(u,v)=-a^y(v,-u-v)$, applying $\sigma$ to this formula, we have
    \begin{equation*}
        \sigma(F^u)(u,v)=-a^y(u,-u-v).
    \end{equation*}
We study the difference between $F_u$ and $\sigma(F^u)$, 
\begin{equation*}
    E:=F_u-\sigma(F^u)=a^y(u,-u-v)-a_y(v,-v-u)
\end{equation*}

Now we set $H:=[F,v]+[\sigma(F),u]=Fv-vF+\sigma(F)u-u(\sigma (F))$, for a word $p$ of degree $\mathrm{deg}(F)-1$, the coefficient of $vpu$ is
    \begin{align*}
        (H\mid vpu)&=-(F\mid pu)+(\sigma(F)\mid vp)\\
        &=-(F\mid pu)+(F\mid u\sigma(p))=-(E\mid p)
    \end{align*}

    Similarly, we have
    \begin{align*}
     (H\mid upv)&=(F\mid up)-(\sigma(F)\mid pv)\\
                &=(F\mid up)-(F\mid \sigma(p)u)=-(E\mid \sigma(p)),
    \end{align*}

\begin{Lemma}\label{lem:E_value}
    Suppose that $S_*(a_*)=-a_*$, then we have
    \begin{equation*}
        E=(-1)^{n-1}\frac{1+(-1)^n}{n}(a\mid x^{n-1}y)(u+v)^{n-1}
    \end{equation*}
\end{Lemma}

\begin{proof}
    We substitute $x=v$, $y=-u-v$ and $z=u$ into the formula of the Proposition \ref{Prop:formula}. 
\end{proof}

\begin{Prop}\label{prop:1-2}
    Suppose that $S_*(a_*)=-a_*$, then we have
    \begin{equation*}
    [a(x,y),x]+[a(z,y),z]=0
    \end{equation*}
    and $(a\mid x^{n-1}y)=0$ for $n$ even.
\end{Prop}

\begin{proof}
    By the Lemma \ref{lem:E_value}, we have that for any word $p$ of degree $\deg(F)-1$, we have
    \begin{align*}
    & (H\mid upv)=-(E\mid \sigma(p))=(-1)^{n}\frac{1+(-1)^n}{n}(a\mid x^{n-1}y),\\
    &(H\mid vpu)=-(E\mid p)=(-1)^{n}\frac{1+(-1)^n}{n}(a\mid x^{n-1}y),
    \end{align*}
    therefore every word in $H$ that beginning and ending in different letters has the same coefficient by the Lemma \ref{lem:mix_boundary}, $H=0$ and $(a\mid x^{n-1}y)=0$ for $n$ even.

    Then we substitute $u=-x-y$, $v=x$ into $H$, we have the equation $[a(x,y),x]+[a(z,y),z]=0.$
\end{proof}

\subsection{From (2) to (1)}

\begin{Lemma}\label{lem:2-1-1}
For $a\in \mathbb{L}^{\ge 3}(x,y)$, suppose that $[a(x,y),x]+[a(z,y),z]=0$, then we have $a^y(z,y)=a_y(x,y)$.
\end{Lemma}

\begin{proof}
    By the change of variable, $u=-x-y,$ $v=x$, we have that $H:=[F,v]+[\sigma(F),u]=0$, where $F(u,v)$ is defined in \eqref{eq:F_uv}.  Therefore for every word $p$, $(H\mid vpu)=-(E\mid p)=0$, which implies $E=0$ and it is $a^y(z,y)=a_y(x,y).$
\end{proof}

Recall that for the algebra $A$, the \emph{weight} of a word is its length, and its depth is its number of occurrences of $y$. 

\begin{Lemma}\label{lem:2-1-2}
    The depth 1 part of $[a(x,y),x]+[a(z,y),z]$ is $-(a\mid x^{n-1}y)(1+(-1)^n)\ad^n_x(y)$.
\end{Lemma}

\begin{proof}
The homogeneous degree $n$, depth 1 part of  $\mathbb{L}(x,y)$ is one dimensional with basis $\ad^{n-1}_xy$, we have that the depth 1 part of $a(x,y)$ is $(a\mid x^{n-1}y)\ad^{n-1}_x y$. The depth 1 part of $a(z,y)=(-1)^{n-1}(a\mid x^{n-1}y)\ad^{n-1}_x(y)$. Therefore the depth 1 part of $[a(x,y),x]=-(a\mid x^{n-1}y)\ad^n_x(y).$ The depth 1 part of $[a(z,y),z]=(-1)^{n-1}(a\mid x^{n-1}y)\ad^n_x(y).$

    So in total, the depth 1 part of the equation is $-(a\mid x^{n-1}y)(1+(-1)^n)\ad^n_x(y)$.
\end{proof}

\begin{Prop}
    Suppose that $[a(x,y),x]+[a(z,y),z]=0$, then we have $S_*(a_*)=-a_*.$
\end{Prop}

\begin{proof}
By the Proposition \ref{Prop:formula}, it suffice to prove that $a^y(z,y)-a_y(x,y)=\frac{1+(-1)^n}{n}(a\mid x^{n-1}y)y^{n-1}$.   

By the Lemma \ref{lem:2-1-2}, the condition $[a(x,y),x]+[a(z,y),z]=0$ implies $-(a\mid x^{n-1}y)(1+(-1)^n)=0$, therefore it suffice to prove $a^y(z,y)-a_y(x,y)=0$ which is implied by the Lemma \ref{lem:2-1-1}.
\end{proof}

\section{Double shuffle Lie algebra and symmetric Kashiwara--Vergne Lie algebra}\label{sec:dmr_krv}

We first recall the definition of the Lie algebra $\mathfrak{dmr}_0$.

\begin{Def}[\cite{Racinet2002}]
Let $\psi\in \mathbb{L}^{\ge 3}(x,y)$,  it is in $\mathfrak{dmr}_0$ if it satisfies
\begin{equation*}
\Delta_*(\psi_*)=\psi_*\otimes 1+1\otimes \psi_*
\end{equation*}    
\end{Def}

\begin{Th}[Same as Theorem \ref{th:double_shuffle}]
    For every $\psi\in \mathfrak{dmr}_0$, it satisfies the equation \[[\psi(x,y),x]+[\psi(-x-y,y),-x-y]=0.\]
\end{Th}

\begin{proof}
By the property of Hopf algebra, suppose that $\Delta_*(a_*)=a_*\otimes 1+1\otimes a_*$, it follows that $S_*(a_*)=-a_*$, then the result follows from Theorem \ref{th: harmonic_antipode}.    
\end{proof}

\subsection{Kashiwara--Vergne Lie algebra}
Now we recall the definition of Kashiwara--Vergne Lie algebra $\mathfrak{krv}_2$ and the symmetric Kashiwara--Vergne Lie algebra $\mathfrak{krv}^{\mathrm{sym}}_2$ introduced in the paper \cite{Alekseev2012} and some reformulation of the $\mathfrak{krv}_2$ in \cite{GTgenus0}.

Let $\mathrm{Der}(A)$ denote the Lie algebra of derivations of $A$, define the vector space $\mathrm{tDer}(A):=A\oplus A$, and for $(\alpha,\beta)\in \mathrm{tDer}(A)$, we denote by $\rho(\alpha,\beta)$, the derivation in $\mathrm{Der}(A)$ determined by 
\begin{equation*}
    \rho(\alpha,\beta)(x)=[x,\alpha],\quad \rho(\alpha,\beta)(y)=[y,\beta].
\end{equation*}
The space $\mathrm{tDer}(A)$ of tangential derivations is a Lie algebra, the Lie bracket is given by
\begin{equation}\label{eq:Lie_bracket}
    [(\alpha_1,\beta_1),(\alpha_2,\beta_2)]=(\alpha_3,\beta_3)
\end{equation}
where
\begin{align*}
 &\alpha_3=\rho(\alpha_1,\beta_1)(\alpha_2)-\rho(\alpha_2,\beta_2)(\alpha_1)+[\alpha_1,\alpha_2],\\
 &\beta_3=\rho(\alpha_1,\beta_1)(\beta_2)-\rho(\alpha_2,\beta_2)(\beta_1)+[\beta_1,\beta_2],
\end{align*}
with this bracket, the map $\rho:\mathrm{tDer}(A)\to \mathrm{Der}(A)$ is a Lie algebra morphism.

A tangetial derivation $(\alpha,\beta)$ is called \emph{special} if we have
\begin{equation*}
    [x,\alpha]+[y,\beta]=0.
\end{equation*}
We denote also the Lie subalgebra of special derivations by
\begin{equation*}
    \mathrm{sDer}(A)=\{(\alpha,\beta)\in \mathrm{tDer}(A)\mid [x,\alpha]+[y,\beta]=0\}.
\end{equation*}

Similarly, we could define $\mathfrak{tder}_2=\mathbb{L}(x,y)\oplus \mathbb{L}(x,y)$ and $\mathfrak{sder}_2=\{(a,b)\in \mathfrak{tder}_2\mid [x,a]+[y,b]=0\}$.

Define $|A|:=A/[A,A]$ and consider the projection map $|-|:A\to A/[A,A]$, the divergence map is
\begin{align*}
    \mathrm{div}: \mathrm{tDer}(A)\to |A|, \quad (a,b)\mapsto |xa_x+yb_y|.
\end{align*}

\begin{Prop}[{\cite[Theorem 4.15, Proposition 8.5]{GTgenus0}}]\label{prop:krv_ref} The Kashiwara--Vergne Lie algebra $(\mathfrak{krv}_2)_{\ge 3}$ consists of $(a,b)\in \mathfrak{sder}_2$ and $a,b\in \mathbb{L}^{\ge 3}(x,y)$ such that
\begin{equation}
    \mathrm{div}((a,b))=|f_0(z)|+|f_1(x)|+|f_2(y)|,
\end{equation}
for some formal power series $f_0(s),f_1(s),f_2(s)\in \mathbb{K}[[s]]$.
    
\end{Prop}

Let $\tau$ be the automorphism of $A$ that exchange $x$ and $y$, we could define the involution 
\begin{equation*}
\Theta^{\tau}:\mathfrak{sder}_2\to \mathfrak{sder}_2, (a,b)\mapsto (\tau(b),\tau(a)).  
\end{equation*}

\begin{Lemma}[{\cite[Section 8.3]{Alekseev2012}}]\label{lem:S3KV}
    $\Theta^{\tau}$ is a Lie algebra involution of $(\mathfrak{krv}_2)_{\ge 3}$.
\end{Lemma}
The symmetric Kashiwara--Vergne Lie algebra is defined in \cite{Alekseev2012} as the fixed point of the involution,
\begin{equation*}
    \mathfrak{krv}^{\mathrm{sym}}_2:=(\mathfrak{krv}_2)^{\Theta^{\tau}}.
\end{equation*}

\subsection{Another involution on $\mathfrak{krv}_2$}

We introduce the following automorphism of $A$,
\begin{equation*}
    \chi:A\to A,\quad x\mapsto
     x,y\mapsto z,
\end{equation*}
by definition, we have $\chi^2=\id$.

\begin{Lemma}\label{lem:inv_1}
   For $P=\varepsilon(P)+P_xx+P_yy=\varepsilon(P)+xP^x+yP^y\in A$, we have the following change of variable formulas,
    \begin{align*}
        & \chi(P)_x=\chi(P_x)-\chi(P_y),\quad \chi(P)_y=-\chi(P_y). \\
        & \chi(P)^x=\chi(P^x)-\chi(P^y),\quad \chi(P)^y=-\chi(P^y).
    \end{align*}
\end{Lemma}
\begin{proof}
    We have that
    \begin{align*}
&\chi(P)=\varepsilon(P)+\chi(P_x)x+\chi(P_y)z=\varepsilon(P)+(\chi(P_x)-\chi(P_y))x+(-\chi(P_y)y).\\
&\chi(P)=\varepsilon(P)+x\chi(P^x)+z\chi(P^y)=\varepsilon(P)+x(\chi(P^x)-\chi(P^y))+y(-\chi(P^y)).
    \end{align*}
\end{proof}

We now define another linear map on $\mathfrak{tder}_2$,
\begin{equation*}
    \Theta^{\chi}:\mathfrak{tder}_2\to \mathfrak{tder}_2,\quad (a,b)\mapsto (\chi(a)-\chi(b),-\chi(b)).
\end{equation*}

The linear map has the following basic properties.

\begin{Lemma}\label{lem:inv_2}
    The map $\Theta^{\chi}$ satisfies the following properties.
    \begin{enumerate}
        \item $\rho(\Theta^{\chi}(a,b))=\chi\circ \rho(a,b)\circ\chi+\ad_{\chi(b)}$, for $(a,b)\in \mathfrak{sder}_2$
        \item $\Theta^{\chi}(\mathfrak{sder}_2)\subset \mathfrak{sder}_2$.
        \item  $(\Theta^{\chi})^2=\id$.
        \item  $\Theta^{\chi}|_{\mathfrak{sder}_2}$ is a Lie algebra morphism.
    \end{enumerate}
\end{Lemma}

\begin{proof}
    For the property (1),  we have that the derivation on both sides agrees on the generators, that is
    \begin{align*}
     &\left(\chi\circ \rho(a,b)\circ\chi+\ad_{\chi(b)}\right) (x)=[x,\chi(a)]+[\chi(b),x]=[x,\chi(a)-\chi(b)]\\
     &\left(\chi\circ \rho(a,b)\circ\chi+\ad_{\chi(b)}\right) (y)=[y,-\chi(b)]\\
    \end{align*}

    For the property (2), suppose that $[x,a]+[y,b]=0$, apply $\chi$, we have $[x,\chi(a)]+[z,\chi(b)]=0$, we could directly compute
    \begin{align*}
        \rho(\Theta^{\chi}(a,b))(z)&=-[x,\chi(a)-\chi(b)]-[y,-\chi(b)]\\
        &=-[x,\chi(a)]-[z,\chi(b)]=0
    \end{align*}
    therefore we have $\Theta^{\chi}(a,b)\in  \mathfrak{sder}_2.$

    For the property (3), we have that
    \begin{align*}
    \Theta^{\chi}\circ\Theta^{\chi}(a,b)=\Theta^{\chi}(\chi(a)-\chi(b),-\chi(b))=(a-b+b,b)=(a,b)
    \end{align*}

    For the Property (4), we use the formula in Property (1) to calculate the Lie bracket \eqref{eq:Lie_bracket} explicitly, the result then follows.
\end{proof}

\begin{Prop}\label{prop:changevariable}
    For any $(a,b)\in \mathfrak{sder}_2$ and $a,b\in \mathbb{L}^{\ge 3}(x,y)$, we have
    \begin{equation}
        \mathrm{div}((\Theta^{\chi}(a,b)))=\chi(\mathrm{div}(a,b)).
    \end{equation}
\end{Prop}

\begin{proof}
For the LHS, we have
\begin{align*}
 \mathrm{div}((\Theta^{\chi}(a,b)))&=\mathrm{div}\left((\chi(a)-\chi(b),-\chi(b))\right)\\
 &=|x(\chi(a)-\chi(b))_x-y\chi(b)_y|\\
 &=|x\chi(a_x)-x\chi(a_y)-x\chi(b_x)+x\chi(b_y)+y\chi(b_y)|
\end{align*}
where the last equality follows from the Lemma \ref{lem:inv_1}.

For the RHS, we have
\begin{align*}
\chi(\mathrm{div}(a,b))=\chi|xa_x+yb_y|=|x\chi(a_x)+z\chi(b_y)|.
\end{align*}
Combining the two sides, we have the result
\begin{align}\label{eq:changevariable_1}
\mathrm{div}((\Theta^{\chi}(a,b)))- \chi(\mathrm{div}(a,b))=|-x\chi(a_y)+(x+y)\chi(b_y)|+|-x\chi(b_x)+(x+y)\chi(b_y)|  
\end{align}

By the decomposition $b=b_xx+b_yy$, applying $\chi$, we have $\chi(b)=\chi(b_x)x+\chi(b_y)z$, then it implies
\begin{equation}\label{eq:changevariable_2}
    |\chi(b)|=|\chi(b_x)x+\chi(b_y)z|=|x\chi(b_x)-(x+y)\chi(b_y)|
\end{equation}
and similarly by the special property, we have that $[x,a]+[y,b]=0$, taking the part of this equality that ends in $y$, we have $b=xa_y+yb_y$, we have
\begin{equation}\label{eq:changevariable_3}
    |\chi(b)|=|x\chi(a_y)-(x+y)\chi(b_y)|.
\end{equation}
Using the relations \eqref{eq:changevariable_2} and \eqref{eq:changevariable_3} in \eqref{eq:changevariable_1}, we obtain the following equality,
\begin{equation}\label{eq:changevariable_4}
\mathrm{div}((\Theta^{\chi}(a,b)))- \chi(\mathrm{div}(a,b))=-2|\chi(b)|,
\end{equation}
as $b\in \mathbb{L}^{\ge 3}(x,y)$, we have that $|\chi(b)|=0$. Therefore the change of variable formula is proved.
\end{proof}

\begin{Prop}\label{prop:Theta_chi}
$\Theta^{\chi}$ is a Lie algebra involution of $(\mathfrak{krv}_2)_{\ge 3}$.
    
\end{Prop}

\begin{proof}
By the Lemma \ref{lem:inv_2},
it suffices to prove for every $(a,b)\in \mathfrak{krv}_2$ and $a,b\in \mathbb{L}^{\ge 3}(x,y)$, $\Theta^{\chi}\left((a,b)\right)\in \mathfrak{krv}_2$.
    By the definition of $\mathfrak{krv}_2$ in Proposition \ref{prop:krv_ref}, \begin{equation}
    \mathrm{div}((a,b))=|f_0(z)|+|f_1(x)|+|f_2(y)|,
\end{equation}
for some formal power series $f_0(s),f_1(s),f_2(s)\in \mathbb{K}[[s]]$. Then by the change variable formula in Proposition \ref{prop:changevariable}, we have
\begin{equation*}
\mathrm{div}\left((\Theta^{\chi}(a,b))\right)=\chi(\mathrm{div}(a,b))=\chi( |f_0(z)|+|f_1(x)|+|f_2(y))|)=|f_0(y)|+|f_1(x)|+|f_2(z)|   
\end{equation*}
which conclude that $\Theta^{\chi}(a,b)$ is in $\mathfrak{krv}_2$ by the Proposition \ref{prop:krv_ref}.
\end{proof}

The following theorem is mentioned without proof in Remark 16.10 in \cite{EF4}, we present a proof for completeness.

\begin{Th}\label{th:S_3KV}
The symmetric group $S_3$ acts on $(\mathfrak{krv}_2)_{\ge 3}$, the action is generated by $\Theta^{\tau}$ and $\Theta^{\chi}$.
\end{Th}

\begin{proof}
    By the Lemma \ref{lem:S3KV} and Proposition \ref{prop:Theta_chi}, $\Theta^{\tau}$ and $\Theta^{\chi}$ are well defined, and $(\Theta^{\tau})^2=(\Theta^{\chi})^2=\id$. Let $r=\tau\chi\tau=\chi\tau\chi$, then $(\Theta^{\tau}\circ \Theta^{\chi}\circ \Theta^{\tau})$ and $(\Theta^{\chi}\circ \Theta^{\tau}\circ \Theta^{\chi})$ send $(a,b)$ to $(-r(a),r(b)-r(a))$, therefore 
    $$(\Theta^{\tau}\circ \Theta^{\chi}\circ \Theta^{\tau})=(\Theta^{\chi}\circ \Theta^{\tau}\circ \Theta^{\chi}).$$ Then there is a group morphism from $S_3$ to automorphism group of $(\mathfrak{krv}_2)_{\ge 3}$ by sending the permutation $(12)$ to $\Theta^{\tau}$ and the permutation $(23)$ to $\Theta^{\chi}$.
\end{proof}

\subsection{Proof of the Theorem \ref{th:symmetr}}

Recall that by the Theorem \ref{th:EFSH}, for $\psi\in \mathfrak{dmr}_0$, there exists unique $b_{\psi}\in \mathbb{L}^{\ge 3}(x,y)$ such that 
\begin{equation}\label{eq:special}
[\psi(x,y),x]+[b_{\psi}(x,y),-x-y]=0,
\end{equation}
and the map $I$ is an injective Lie algebra morphism
\begin{equation*}
    I: \mathfrak{dmr}_0\to \mathfrak{krv}_2, \psi(x,y)\mapsto (\psi(z,y)-b_{\psi}(z,y),\psi(z,y)).
\end{equation*}

\begin{Prop}
For $\psi\in \mathfrak{dmr}_0$, $b_{\psi}(x,y)=\psi(z,y)$ and the map in Theorem \ref{th:EFSH} becomes
\begin{equation*}
    I: \mathfrak{dmr}_0\to \mathfrak{krv}_2, \psi(x,y)\mapsto (\psi(z,y)-\psi(x,y),\psi(z,y)).
\end{equation*}
\end{Prop}

\begin{proof}
    By the Theorem \ref{th:double_shuffle}, we know that $[\psi(x,y),x]+[\psi(-x-y,y),-x-y]=0$, because $b_{\psi}$ is unique, it is equal to $\psi(z,y)$.
\end{proof}

\begin{Th}[Same as Theorem \ref{th:symmetr}]
    The map $J:=\Theta^{\chi}\circ I$, 
    \begin{equation}
    J:\mathfrak{dmr}_0\to \mathfrak{krv}^{\mathrm{sym}}_2,\quad \psi(x,y)\mapsto (-\psi(x,-x-y),-\psi(y,-x-y))
    \end{equation}
    is an injective Lie algebra morphism.
\end{Th}

\begin{proof}
    By the Proposition \ref{prop:Theta_chi}, we have that $J$ is a well defined map from $\mathfrak{dmr}_0$ to $\mathfrak{krv}_2$, and 
    \begin{equation}
        J(\psi(x,y))=\Theta^{\chi}\left((\psi(z,y)-\psi(x,y),\psi(z,y)    \right))=\left(-\psi(x,z),-\psi(y,z)\right)
    \end{equation}
    and we have $\Theta^{\tau}(-\psi(x,z),-\psi(y,z))=(-\psi(x,z),-\psi(y,z))$, therefore the image is in $\mathfrak{krv}^{\mathrm{sym}}_2$. The map $J$ is a Lie algebra map because $I$ is a Lie algebra morphism and $\Theta^{\chi}$ is a Lie algebra morphism by the Lemma \ref{lem:inv_2}. The injectivity is because $\Theta^{\chi}$ is involutive.
\end{proof}

\bibliographystyle{abbrv}
\bibliography{double_shuffle}

\end{document}